\documentclass[12pt]{amsart}
\usepackage{amsfonts, amssymb, amsmath, amsthm}
\usepackage[mathscr]{eucal}
\usepackage[hmargin=1.2in, vmargin=1.5in]{geometry}
\input amssym.def
\input amssym

\newtheorem{theorem}{Theorem}[section]
\newtheorem{proposition}[theorem]{Proposition}
\newtheorem{lemma}[theorem]{Lemma}
\newtheorem{corollary}[theorem]{Corollary}

\theoremstyle{definition}
\newtheorem{definition} [theorem]{Definition}

\theoremstyle{remark} \newtheorem{remark} {Remark}

\numberwithin{equation}{section}

\begin{document}

\title[The ridgelet transform and quasiasymptotic behavior of distributions]{The ridgelet transform and quasiasymptotic behavior of distributions}

\author[S. Kostadinova]{Sanja Kostadinova}
\address{Faculty of Electrical Engineering and Information Technologies, Ss. Cyril and Methodius University, Rugjer Boshkovik bb, 1000 Skopje, Macedonia}
\email{ksanja@feit.ukim.edu.mk}

\author[S. Pilipovi\'{c}]{Stevan Pilipovi\'{c}}
\address{Department of Mathematics, University of Novi Sad, Trg Dositeja Obradovi\'ca 4, Novi Sad, Serbia}
\email {stevan.pilipovic@dmi.uns.ac.rs}
\thanks{This work was supported by the Serbian Ministry of Education, Science and Technological Development, through the project \# 174024.}

\author[K. Saneva]{Katerina Saneva}
\address{Faculty of Electrical
Engineering and Information Technologies, Ss. Cyril
and Methodius University, Karpos 2 bb, 1000 Skopje, Macedonia}
\email{saneva@feit.ukim.edu.mk}

\author[J. Vindas]{Jasson Vindas}
\thanks{J. Vindas gratefully acknowledges support by Ghent University, through the BOF-grant 01N01014.}
\address{Ghent University, Krijgslaan 281 Gebouw S22, B-9000 Gent, Belgium}
\email{jvindas@cage.UGent.be}


\begin{abstract}
We characterize the quasiasymptotic behavior of distributions in terms of a Tauberian theorem for ridgelet transforms.
\end{abstract}

\maketitle

\section{Introduction}
Ridgelet analysis may be considered as an adaptation of wavelet analysis for dealing with higher dimensional phenomena \cite{candes3}. The theory of the continuous ridgelet transform for functions was developed by Cand\`{e}s in \cite{candes1,candes2}.  This transform is the composition of the Radon transform with a one-dimensional continuous wavelet transform. Thus, ridgelet analysis can be seen as a form of wavelet analysis performed in the Radon domain. In \cite{KPSV}, the authors have extended the theory to include ridgelet transforms of Lizorkin distributions, that is, elements of $\mathcal{S}' _{0}(\mathbb{R}^{n})$, the dual of the space of highly time-frequency localized test functions $\mathcal{S} _{0}(\mathbb{R}^{n})$ \cite{hols}.

The purpose of this paper is to study the quasiasymptotic behavior of Lizorkin distributions via ridgelet analysis. The quasiasymptotic behavior was introduced by Zav'yalov in the context of  quantum field theory and it was further studied by him, Vladimirov and Drozhzhinov in connection with Tauberian theorems for multidimensional Laplace transforms \cite{VDZ}. This concept measures the scaling asymptotic properties of distributions through asymptotic comparison with Karamata regularly varying functions. The main result of this article (Theorem \ref{tauberRT}) is a characterization of the quasiasymptotic behavior in terms of a Tauberian theorem for the ridgelet transform. We point out that there is an extensive literature in Abelian and Tauberian theorems for generalized functions; see, e.g., the monographs \cite{ML,PST,PSV,VDZ} and references therein for the analysis of various integral transforms. For studies involving the quasiasymptotic behavior and wavelet analysis we refer to \cite{DZ1,DZ2,KV,PTT,PT,prof44,saneva1,saneva2,SV,vindas3,Walter2}.

Most of our arguments in this article rely on the intrinsic connection between the ridgelet, Radon, and wavelet transforms. For distributions, such a connection must be carefully handled and involves ideas from the theory of tensor products of topological vector spaces. Section \ref{preli} collects background material from \cite{KPSV} on these three integral transforms. In Section \ref{section bounded sets}, we present a ridgelet transform characterization of the bounded subsets of $\mathcal{S}'_{0}(\mathbb{R}^{n})$; we also show in this section that the Radon transform on $\mathcal{S}'_{0}(\mathbb{R}^{n})$ is a topological isomorphism into its range. It is interesting to notice that the Radon transform may fail to have the latter property even on spaces of test functions; for instance, Hertle has shown \cite{hertle2} that the Radon transform on $\mathcal{D}(\mathbb{R}^{n})$ is not an isomorphism of topological vector spaces into its range. Finally, Section \ref{section abel-tauber} deals with Abelian and Tauberian theorems for the ridgelet transform. Theorem \ref{tauberRT} should be compared with the Taberian theorems for wavelet transforms from \cite{DZ2,prof44,vindas3}.

\section{Preliminaries}\label{preli}
\subsection{Spaces}\label{spaces}
We denote as $\mathbb H=\mathbb R\times \mathbb R_{+}$ the upper-half plane and ${\mathbb Y^{n+1}}=\mathbb{S}^{n-1}\times\mathbb{H}=\{(\mathbf{u},b,a):\: \mathbf{u} \in {\mathbb S^{n-1}}, b
\in {{\mathbb R}}, a\in{\mathbb R}_{+} \}$,
where ${\mathbb S^{n-1}}$ stands for the unit sphere of $\mathbb R^{n}$. We always assume that the dimension $n\geq2$.

We provide all distribution spaces with the strong dual topologies. The Schwartz spaces ${\mathcal D}({{\mathbb R}}^n)$, ${\mathcal S}({{\mathbb R}}^n)$, ${\mathcal D}'({{\mathbb
R}}^n)$, ${\mathcal S}'({{\mathbb R}}^{n})$, and $\mathcal{D}'_{ L^{1}}(\mathbb{R}^{n})$ are well known \cite{schwartz}. We will also work with the Lizorkin test function space ${{\mathcal S}}_0({{\mathbb
R}}^n)$ of highly time-frequency localized functions over ${{\mathbb
R}}^n$ \cite{hols}. It consists of those elements of ${\mathcal S}({{\mathbb R}}^n)$ having
all moments equal to $0$, namely, $\phi\in {{\mathcal S}}_0({\mathbb{R}}^{n})$ if $
\int_{{{\mathbb R}}^n}{\mathbf x^m\phi(\mathbf x)d\mathbf x}=0, $ for all $m\in {{\mathbb N}}_0^n.$ It is a closed subspace of  ${\mathcal S}({{\mathbb R}}^n)$.
Its dual space $\mathcal{S}'_{0}(\mathbb{R}^{n})$, known as the space of Lizorkin distributions, is canonically isomorphic to the quotient of $\mathcal{S}'(\mathbb{R}^{n})$ by the space of polynomials.
We denote by $\mathcal{D}(\mathbb{S}^{n-1})$ the space of smooth functions on the sphere. Given a locally convex space $\mathcal{A}$ of smooth test functions on $\mathbb{R}$, we write $\mathcal{A}(\mathbb{S}^{n-1}\times\mathbb{R})$ for the space of functions $\varrho(\mathbf{u},p)$ having the properties of $\mathcal{A}$ in the variable $p\in\mathbb{R}$ and being smooth in $\mathbf{u}\in\mathbb{S}^{n-1}$.

We introduce ${\mathcal S}({\mathbb Y^{n+1}})$ as the space of functions ${\Phi }\in C^{\infty
}({\mathbb Y^{n+1}})$ satisfying the decay conditions
\begin{equation} \label{eqNorms}
\rho_{s,r}^{l,m,k}(\Phi)={\mathop{\sup }_{(\mathbf{u},b,a)\in {\mathbb Y^{n+1}}}
\left(a+\frac{1}{a}\right)^{s}{{{\rm (1+}|b|{\rm )}}^{r}}
\left|\frac{{\partial
}^l}{\partial a^l}\frac{{\partial }^m}{\partial
b^m}{\triangle_{\mathbf{u}}^{k}}{ \Phi
}\left(\mathbf{u},b,a \right)\right|{\rm \ }<\infty {\rm \ }\ }
\end{equation}
for all $l,m,k,s,r \in {{\mathbb N}}_0$, where ${\triangle_{\mathbf{u}}}$ is the Laplace-Beltrami operator on the unit sphere $\mathbb S^{n-1}$. The functions from ${\mathcal S}({\mathbb Y^{n+1}})$ thus have fast decay for large and small values of the scale variable $a$. The topology of this space is defined by means of the seminorms (\ref{eqNorms}). Its dual is denoted by ${\mathcal S}'({\mathbb Y^{n+1}})$.

A related space is $\mathcal{S}(\mathbb{H})$, the space of highly localized test functions on the upper half-plane \cite{hols}. Its elements are smooth functions $\Psi$ on $\mathbb{H}$ that satisfy
\[\mathop{\sup }_{(b,a)\in {\mathbb H}}\left(a+\frac{1}{a}\right)^{s}
(1+|b|)^{r}
\left|\frac{{\partial
}^l}{\partial a^l}\frac{{\partial }^m}{\partial
b^m}{ \Psi
}\left(b,a \right)\right|{\rm \ }<\infty,
\]
for all $l,m,s,r \in {{\mathbb N}}_0$; its topology being defined in the canonical way \cite{hols}.

Observe that the nuclearity of the Schwartz spaces \cite{treves} immediately yields the equalities $\mathcal{S}(\mathbb{Y}^{n+1})=\mathcal{D}(\mathbb{S}^{n-1})\hat{\otimes}\mathcal{S}(\mathbb{H})$ and $\mathcal{S}_{0}(\mathbb{S}^{n-1}\times \mathbb{R})=\mathcal{D}(\mathbb{S}^{n-1})\hat{\otimes}\mathcal{S}_{0}(\mathbb{R})$, where $X\hat{\otimes}Y$ is the topological tensor product space obtained as the completion of $X\otimes Y$ in, say, the $\pi$-topology  or, equivalently in these cases, the $\varepsilon$-topology \cite{treves}. We therefore have the following isomorphisms $\mathcal{S}'(\mathbb{Y}^{n+1})\cong \mathcal{S}'(\mathbb {H}, \mathcal {D}'(\mathbb S^{n-1}))\cong \mathcal{D}'(\mathbb{S}^{n-1},\mathcal{S}'(\mathbb{H}))$, the very last two spaces being spaces of vector-valued distributions \cite{schwartzv,silva,treves}. We shall identify these three spaces and write
\begin{equation}
\label{tensors}
\mathcal{S}'(\mathbb{Y}^{n+1})= \mathcal{S}'(\mathbb {H}, \mathcal {D}'(\mathbb S^{n-1}))=\mathcal{D}'(\mathbb{S}^{n-1},\mathcal{S}'(\mathbb{H})).
\end{equation}
The equality (\ref{tensors}) being realized via the standard identification
\begin{equation}
\label{tensors2}
\left\langle F,\varphi\otimes \Psi \right\rangle=\left\langle\left\langle F,\Psi \right\rangle,\varphi\right\rangle=  \left\langle \left\langle F,\varphi \right\rangle,\Psi\right\rangle, \     \    \ \Psi\in\mathcal{S}(\mathbb{H}),\ \varphi\in\mathcal{D}(\mathbb{S}^{n-1}).
\end{equation}
Likewise, we have the right to write
$$\mathcal{S}_{0}'(\mathbb{S}^{n-1}\times \mathbb{R})=\mathcal{S}_{0}'(\mathbb{R},\mathcal{D}'(\mathbb{S}^{n-1}))=\mathcal{D}'(\mathbb{S}^{n-1},\mathcal{S}'_{0}(\mathbb{R})).$$

We shall say that $F\in\mathcal{S}'(\mathbb{Y}^{n+1})$ is a function of slow growth in the variables $(b,a)\in\mathbb{H}$ if
$
\left\langle F(\mathbf{u},b,a),\varphi(\mathbf{u})\right\rangle_{\mathbf{u}}
$
is such for every $\varphi\in\mathcal{D}(\mathbb{S}^{n-1})$, namely, it is a function that satisfies the bound \begin{equation*}|\left\langle F(\mathbf{u},b,a),\varphi(\mathbf{u})\right\rangle_{\mathbf{u}} |\leq C\left(a+\frac{1}{a}\right)^{s}(1+|b|)^{s}, \  \  \  (b,a)\in\mathbb{H},\end{equation*}
for some positive constants $C=C_{\varphi}$ and $s=s_{\varphi}$.

\subsection{The ridgelet transform}\label{ridgelet transform functions}

Let $\psi \in {\mathcal S}({\mathbb R})$. For $\left(\mathbf{u},b,a \right)\in {\mathbb Y^{n+1}}$, where $\mathbf{u}$ is the orientation parameter, $b$ is the location parameter, and $a$ is the scale parameter, we define the  function ${\psi
}_{\mathbf{u},b,a }:{{\mathbb R}}^n\to {\mathbb C}$, called
\emph{ridgelet}, as

\[{\psi }_{\mathbf{u},b,a }\left(\mathbf{x}\right)=
\frac{1}{a}\psi
\left(\frac{\mathbf x\cdot \mathbf u -b}{a}\right),\  \  \ {\mathbf
x}\in {{\mathbb R}}^n.\]

This function is constant along hyperplanes $\mathbf x\cdot \mathbf u
= \textnormal{const.}$, called ``ridges".  In the orthogonal direction it is a
wavelet, hence the name ridgelet. The ridgelet transform ${\mathcal R}_{\psi}$ of an integrable function $ f\in
L^1({{\mathbb R}}^n)$ (or an integrable distribution $ f\in
\mathcal{D}'_{L^1}({{\mathbb R}}^n)$) is defined by

\begin{equation}\label{ridgelet}
{{\mathcal R}}_{\psi}f\left(\mathbf{u},b,a \right)=\int_{\mathbb
R^n}{f(\mathbf{x}){\overline{\psi
}_{\mathbf{u},b,a}}(\mathbf{x})d\mathbf{x}}=\left\langle f(\mathbf x),{\overline{\psi
}_{\mathbf{u},b,a}}(\mathbf{x})\right\rangle_\mathbf{x},\end{equation}
where $\left(\mathbf{u},b,a \right)\in {\mathbb
Y^{n+1}}$.
\begin{definition}\label{nondegenerate} Let $\psi\in\mathcal{S}(\mathbb{R})\setminus\{0\}$. A test function $\eta\in\mathcal{S}(\mathbb{R})$ is said to be a \emph{reconstruction neuronal activation function} for $\psi$ if the constant
\begin{equation*}
K_{\psi, \eta}:=(2\pi)^{n-1}\int^{\infty }_{-\infty}\overline{\widehat{\psi }}(\omega){\widehat{\eta }}(\omega)\frac{d\omega}{|\omega|^{n}}
\end{equation*}
is non-zero and finite.\end{definition}

It is not hard to see that every $\psi\in\mathcal{S}(\mathbb{R})\setminus\{0\}$ has a reconstruction neuronal activation function $\eta$ which may be chosen from $\mathcal{S}_{0}(\mathbb{R})$.

If $\psi$ and $\eta$ are as in Definition \ref{nondegenerate} and if $f\in L^{1}(\mathbb{R}^{n})$ is such that  $\widehat{f}\in L^{1}(\mathbb{R}^{n})$, then the following reconstruction formula  holds pointwisely \cite[Prop. 3.2]{KPSV},

\begin{equation} \label{reconstruction1}
f\left(\mathbf{x}\right)=\frac{1}{K_{\psi,\eta}}
\int_{\mathbb{S}^{n-1}}\int^{\infty
}_{0}\int^{\infty }_{-\infty }
{\mathcal R}_{\psi}f\left(\mathbf{u},b,a
\right){\eta }_{\mathbf{u},b,a}(
\mathbf{x})\frac{dbdad\mathbf{u}}{a^{n}}.
\end{equation}
Given $\psi\in\mathcal{S}(\mathbb{R})$, we introduce the \emph{ridgelet synthesis operator} as

\begin{equation}
\label{synthesis}
\mathcal{R}_{\psi}^{t} \Phi(\mathbf{x}):= \int_{\mathbb{S}^{n-1}}\int^{\infty
}_{0}\int^{\infty }_{-\infty } \Phi(\mathbf{u},b,a){\psi }_{\mathbf{u},b,a}(
\mathbf{x})\frac{dbdad\mathbf{u}}{a^{n}}, \  \  \  \mathbf{x}\in\mathbb{R}^{n}.
\end{equation}
The integral (\ref{synthesis}) is absolutely convergent, for instance, if $\Phi \in \mathcal{S}(\mathbb{Y}^{n+1})$. Observe that the reconstruction formula (\ref{reconstruction1}) can be rewritten as $K_{\psi,\eta} f(\mathbf{x})=(\mathcal{R}^{t}_{\eta} (\mathcal{R}_{\psi} f))(\mathbf{x})$.

We have shown in \cite{KPSV} that the two ridgelet mappings ${{\mathcal R}}_{\psi}:\mathcal S_{0}(\mathbb R^n)\to {\mathcal S}({\mathbb Y^{n+1}})\ $ and ${\mathcal {R}}^{t}_{\psi}:
{\mathcal S}({\mathbb Y^{n+1}})\to {{\mathcal S}}_{0}({{\mathbb R}}^n)$ are  continuous, provided that $\psi\in\mathcal{S}_{0}(\mathbb{R})$. These continuity results allow us to define the ridgelet transform of $f\in\mathcal S'_{0}(\mathbb R^n)$ with respect to $\psi\in\mathcal{S}_{0}(\mathbb{R})$ as the element $\mathcal{R}_{\psi}f\in \mathcal{S}'(\mathbb{Y}^{n+1})$ whose action on test functions is given by
\begin{equation*}
\langle \mathcal{R}_{\psi}{f}, \Phi\rangle:=\langle f, \mathcal{R}^{t}_{\overline{\psi}}{\Phi}\rangle, \ \ \  \Phi \in \mathcal{S}(\mathbb{Y}^{n+1}).\end{equation*}
Moreover, we define the ridgelet synthesis operator $\mathcal{R}^{t}_{\psi} :\mathcal{S}'(\mathbb{Y}^{n+1})\to\mathcal{S}'_{0}({\mathbb{R}^{n}})$ as

\begin{equation*}
\langle \mathcal{R}^{t}_{\psi}F, \phi\rangle:=\langle F, \mathcal{R}_{\overline{\psi}}{\phi}\rangle, \  \  \ F\in\mathcal{S}' (\mathbb{Y}^{n+1}),\ \ \ \phi \in \mathcal S({\mathbb R^{n}}).\end{equation*}
We immediately obtain that the ridgelet transform $\mathcal{R}_{\psi}:\mathcal{S}'_{0}({\mathbb{R}^{n}})\to \mathcal{S}'(\mathbb{Y}^{n+1})$ and the ridgelet synthesis operator $\mathcal{R}^{t}_{\psi} :\mathcal{S}'(\mathbb{Y}^{n+1})\to\mathcal{S}'_{0}({\mathbb{R}^{n}})$ are  continuous linear mappings. In addition \cite[Thrm 5.4]{KPSV}, the following inversion formula holds

\begin{equation}\label{eqidentity}
\operatorname{id}_{\mathcal{S}_{0}'(\mathbb {R}^n)}=\frac{1}{K_{\psi,\eta}}(\mathcal{R}_{\eta}^{t}\circ \mathcal{R_{\psi}}),
\end{equation}
\noindent where $\eta\in\mathcal{S}_{0}(\mathbb{R})$ is a reconstruction neuronal activation function for $\psi\in\mathcal{S}_{0}(\mathbb{R})\setminus\{0\}$.

It is very important to point out that the definition of the distributional ridgelet transform is consistent with (\ref{ridgelet}) for test functions in the following sense. If $f\in L^{1}(\mathbb{R}^{n})$, or more generally $f\in \mathcal{D}'_{L^{1}}(\mathbb{R}^{n})$, the function (\ref{ridgelet}) is continuous and bounded on $\mathbb{Y}^{n+1}$; one can then show \cite[Thrm 5.5]{KPSV} that
\begin{equation}
\label{ridgelet functions}
\langle \mathcal{R}_{\psi}{f}, \Phi\rangle= \int^{\infty
}_{0}\int^{\infty }_{-\infty }\int_{\mathbb{S}^{n-1}} \mathcal{R}_{\psi}f(\mathbf{u},b,a)\Phi(\mathbf{u},b,a) \frac{d\mathbf{u}dbda}{a^n}, \    \    \   \Phi\in\mathcal{S}(\mathbb{Y}^{n+1}).
\end{equation}

\subsection{The Radon transform}
Let $f$ be a function that is integrable on hyperplanes of $\mathbb R^{n}$. For $\mathbf u\in\mathbb{S}^{n-1}$ and $p\in
\mathbb{R}$, the equation $\mathbf{x}\cdot
\mathbf{u}=p$ specifies a hyperplane of $\mathbb R^n$. Then, the Radon transform of $f$ is defined as
\[ Rf(\mathbf{u},p)=Rf_{\mathbf{u}}(p):=\int_{\mathbf{x}\cdot\mathbf{u}=p}{f(\mathbf{x})d\mathbf{x}}.\]
The dual Radon transform (or back-projection) $R^{\ast}\varrho$ of the function $\varrho \in L^{\infty}(\mathbb{S}^{n-1}\times \mathbb{R})$ is defined as
\begin{equation*}
R^{\ast}\varrho(\mathbf{x})=\int_{\mathbb S^{n-1}}\varrho(\mathbf{u},\mathbf{x}\cdot \mathbf{u})d\mathbf{u}.
\end{equation*}
See Helgason's book \cite{helgason} for properties of the Radon transform. It can be shown \cite{KPSV} that the mappings $R:\mathcal{S}_{0}(\mathbb{R}^{n})\to \mathcal{S}_{0}(\mathbb{S}^{n-1}\times\mathbb{R})
$ and $R^{\ast}:\mathcal{S}_{0}(\mathbb{S}^{n-1}\times\mathbb{R})\to\mathcal{S}_{0}(\mathbb{R}^{n})$ are continuous. The first of this mappings is injective, while $R^{\ast}$ is surjective. Therefore, one can also extend the definition of the Radon transform to $\mathcal{S}'_{0}(\mathbb{R}^{n})$ via the formula
\begin{equation*}
\left\langle R f,\varrho \right\rangle=\left\langle f,R^{\ast}\varrho \right\rangle.
\end{equation*}
Clearly, $R:\mathcal{S}'_{0}(\mathbb{R}^n)\to \mathcal{S}'_{0}(\mathbb{S}^{n-1}\times \mathbb{R})$ is continuous and injective.

\subsection{The wavelet transform} Given $f\in\mathcal{S}'(\mathbb{R})$ and $\psi \in \mathcal S(\mathbb R)$ (or $f\in\mathcal{S}'_{0}(\mathbb{R})$ and $\psi \in \mathcal S_{0}(\mathbb R)$), the wavelet transform $\mathcal W_{\psi}f(b,a)$ of $f$ is defined by
\begin{equation*}
\mathcal W_{\psi}f(b,a)= \int_{\mathbb R}f(x) \frac{1}{a}\overline{\psi}\Big(\frac{x-b}{a}\Big)dx=\left\langle  f(x), \frac{1}{a}\overline{\psi}\Big(\frac{x-b}{a}\Big)\right\rangle_{x}, \ \ \ (b,a)\in\mathbb{H}.
\end{equation*}
We refer to Holschneider's book \cite{hols} for a distributional wavelet transform theory based on the spaces $\mathcal{S}_{0}(\mathbb{R})$, $\mathcal{S}(\mathbb{H})$, $\mathcal{S}'_{0}(\mathbb{R})$, and $\mathcal{S}'(\mathbb{H})$. We shall need here the wavelet transform of vector-valued distributions, as explained in \cite[Sect. 5 and 8]{prof44}.

We deal here with wavelet analysis on $\mathcal{S}_{0}(\mathbb{S}^{n-1}\times \mathbb{R})$ and $\mathcal{S}_{0}'(\mathbb{S}^{n-1}\times \mathbb{R})$. Given $\psi\in\mathcal{S}_{0}(\mathbb{R})$, we let $\mathcal{W}_{\psi}$ act on the real variable $p$ of functions (or distributions)  $g(\mathbf{u},p)$, that is,
\begin{equation}\label{waveletS}
\mathcal{W}_{\psi} g(\mathbf{u},b,a):= \int_{-\infty}^{\infty} \frac{1}{a}\overline{\psi}\left(\frac{p-b}{a}\right) g(\mathbf{u},p)dp= \left\langle  g(\mathbf{u},p), \frac{1}{a}\overline{\psi}\Big(\frac{p-b}{a}\Big)\right\rangle_{p},
\end{equation}
$(\mathbf{u},b,a)\in\mathbb{Y}^{n+1}$. Similarly, we define the wavelet synthesis operator on $\mathcal{S}(\mathbb{Y}^{n+1})$ as
\begin{equation*}
\mathcal{M}_{\psi}\Phi(\mathbf{u},p)= \int_{0}^{\infty}\int_{-\infty}^{\infty} \frac{1}{a}\psi\left(\frac{p-b}{a}\right)\Phi(\mathbf{u},b,a)\frac{dbda}{a}.
\end{equation*}
The mappings, \cite[Cor. 4.3]{KPSV}, $\mathcal{W}_{\psi}:\mathcal{S}_0
( \mathbb{S}^{n-1}\times\mathbb{R})\rightarrow \mathcal{S} (\mathbb{Y}^{n + 1})$ and $\mathcal{M}_{\psi}: \mathcal{S}
(\mathbb{Y}^{n+1})\rightarrow \mathcal{S}_{0}
(\mathbb{S}^{n-1}\times\mathbb{R}) $ are continuous. As remarked in Subsection \ref{spaces}, we have $\mathcal{S}_{0}'(\mathbb{S}^{n-1}\times \mathbb{R})=\mathcal{S}_{0}'(\mathbb{R},\mathcal{D}'(\mathbb{S}^{n-1}))$ and $\mathcal{S}'(\mathbb{Y}^{n + 1})=\mathcal{S}'(\mathbb{H}, \mathcal{D}'(\mathbb{S}^{n-1}))$. This allows us to interpret the wavelet transform (\ref{waveletS}),
\[\mathcal{W}_{\psi}:\mathcal{S}_{0}'(\mathbb{S}^{n-1}\times \mathbb{R})=\mathcal{S}_{0}'(\mathbb{R},\mathcal{D}'(\mathbb{S}^{n-1}))\to  \mathcal{S}'(\mathbb {H}, \mathcal {D}'(\mathbb S^{n-1}))=\mathcal{S}'(\mathbb{Y}^{n+1}),\]
as a wavelet transform with values in the DFS space $\mathcal{D}'(\mathbb{S}^{n-1})$. Actually, if $g\in\mathcal{S}_{0}'(\mathbb{S}^{n-1}\times\mathbb{R})$, then $\mathcal{W}_{\psi}g:\mathbb{H}\to \mathcal{D}'(\mathbb{S}^{n-1})$ is a smooth vector-valued function of slow growth on $\mathbb{H}$, whose action on test functions $\Phi\in\mathcal{S}(\mathbb{Y}^{n+1})$ is specified by
\begin{equation}
\label{waveletvector1}
\left\langle \mathcal{W}_{\psi}g,\Phi\right\rangle:=\int_{0}^{\infty}\int_{-\infty}^{\infty} \left\langle \mathcal{W}_{\psi}g(\mathbf{u},b,a),\Phi(\mathbf{u},b,a)\right\rangle_{\mathbf{u}}\frac{dbda}{a}\: .
\end{equation}
Implicit in (\ref{waveletvector1}) is the fact that we are using the measure $a^{-1}dbda$ as the \emph{standard measure} on $\mathbb{H}$ for the identification of functions of slow growth with distributions on $\mathbb{H}$. This choice is the natural one for wavelet analysis, in the sense that one can check that the following duality relation holds:
\begin{equation*}
\left\langle \mathcal{W}_{\psi}g,\Phi\right\rangle=\left\langle g,\mathcal{M}_{\overline{\psi}}\Phi\right\rangle,
\end{equation*}
for all for $g\in\mathcal{S}_{0}'(\mathbb{S}^{n-1}\times\mathbb{R})$ and $\Phi\in\mathcal{S}(\mathbb{Y}^{n+1}).$ (See \cite[Sect. 5 and 8]{prof44} for additional comments on the vector-valued wavelet transform.)

\subsection{Relation between the ridgelet, Radon and wavelet transforms}

The ridgelet transform is intimately connected with the
Radon and wavelet transforms. Changing variables in (\ref{ridgelet}) to $\mathbf{x}=p\mathbf{u}+\mathbf{y}$, where $p\in\mathbb{R}$ and $\mathbf{y}$ runs over the hyperplane perpendicular to $\mathbf{u}$, one readily obtains

\begin{equation}\label{rad-rid}{{\mathcal R}}_{\psi}f\left(\mathbf{u},b,a \right)
=\mathcal W_{\psi}(Rf_{\mathbf{u}})(b,a),
\end{equation}
where $\mathcal{W}_{\psi}$ is a one-dimensional wavelet transform. The relation (\ref{rad-rid}) holds if $f\in L^{1}(\mathbb{R}^{n})$ (or more generally if $f\in\mathcal{D}'_{ L^{1}}(\mathbb{R}^{n})$). Thus, ridgelet analysis can be seen as a form of wavelet analysis
in the Radon domain, i.e., the ridgelet transform is precisely the
application of a one dimensional wavelet transform to the slices
of the Radon transform where $\mathbf{u}$ remains fixed and $p$
varies.

There is also an analog of (\ref{rad-rid}) for $f \in \mathcal S'_{0}(\mathbb R^{n})$. One can show \cite[Thm. 7.1]{KPSV} that if $f \in \mathcal S'_{0}(\mathbb R^{n})$ and $\psi\in\mathcal{S}_{0}(\mathbb{R})$, then
 \begin{equation}
\label{eqdesingular}
\left\langle \mathcal{R}_{\psi}f,\Phi\right\rangle= \int_{0}^{\infty}\int_{-\infty}^{\infty} \left\langle \mathcal{W}_{\psi}(Rf)(\mathbf{u},b,a),\Phi(\mathbf{u},b,a)\right\rangle_{\mathbf{u}} \frac{dbda}{a^{n}}, \  \  \  \Phi\in\mathcal{S}(\mathbb{Y}^{n+1}).
\end{equation}
Furthermore, $\mathcal{R_{\psi}}f\in C^{\infty}(\mathbb H, \mathcal D'(\mathbb S^{n-1}))$ and is of slow growth on $\mathbb{H}$, and, if $\psi\neq0$, the following desingularization formula holds  \cite[Thrm 7.2]{KPSV}.

\begin{equation}
\label{eq desingular 2}
\langle f,\phi\rangle= \frac{1}{K_{\psi,\eta}}\int_{0}^{\infty}\int_{\mathbb{R}} \left\langle \mathcal{W}_{\psi}(Rf)(\mathbf{u},b,a),\mathcal{R}_{\overline{\eta}}\:\phi (\mathbf{u},b,a)\right\rangle_{\mathbf{u}} \frac{dbda}{a^{n}},
\end{equation}
for all $\phi\in\mathcal{S}_{0}(\mathbb{R}^{n})$, where $\eta\in\mathcal{S}_{0}(\mathbb{R})$ is a reconstruction neuronal activation function for $\psi$.

It is interesting to compare (\ref{waveletvector1}) with (\ref{eqdesingular}). Define first the multiplier operators
\begin{equation*}
J_{s}:\mathcal{S}'(\mathbb{Y}^{n+1})\to \mathcal{S}'(\mathbb{Y}^{n+1}), \ \ \ (J_{s}F)(\mathbf{u},b,a)=a^{s}F(\mathbf{u},b,a),  \ \  \ s\in\mathbb{R}.
\end{equation*}
According to (\ref{waveletvector1}), the relation (\ref{eqdesingular}) for distributions might be rewritten as
\begin{equation}
\label{rid-rad-wav}
\mathcal{R}_{\psi}= J_{1-n}\circ\mathcal{W}_{\psi}\circ R.
\end{equation}
Observe that (\ref{rid-rad-wav}) is not in contradiction with (\ref{rad-rid}). Indeed, if $f\in L^{1}(\mathbb{R}^{n})$ (or more generally $f\in \mathcal{D}'_{L^{1}}(\mathbb{R}^{n})$), then (\ref{rad-rid}) expresses an equality between functions, (\ref{eqdesingular}) is then in agreement with (\ref{ridgelet functions}), whereas (\ref{rid-rad-wav}) simply responds to our convention (\ref{waveletvector1}) of using the measure $a^{-1}dbda$ for identifying wavelet transforms with vector-valued distributions on $\mathbb{H}$. We also have to warn the reader that under this convention, the smooth function $F_{\varphi}(b,a)=\langle\mathcal{R_{\psi}}f(\mathbf{u},b,a), \varphi(\mathbf{u}) \rangle_{\textbf{u}}$ from the standard identification (\ref{tensors2}), where $\varphi\in\mathcal{D}(\mathbb{S}^{n-1})$, is the one that satisfies
\begin{equation}
\label{eq extra}
\langle\mathcal{R_{\psi}}f(\mathbf{u},b,a), \varphi(\mathbf{u})\Psi(b,a) \rangle_{\textbf{u}}=\int_{0}^{\infty}\int_{-\infty}^{\infty}F_{\varphi}(b,a)\Psi(b,a)\frac{dbda}{a}, \  \  \ \Psi\in\mathcal{S}(\mathbb{H});
\end{equation}
so that if $f\in\mathcal{D}'_{L^{1}}(\mathbb{R}^{n})$, we have, as pointwise equality between functions,
\begin{equation}
\label{eq extra2}
\langle\mathcal{R_{\psi}}f(\mathbf{u},b,a), \varphi(\mathbf{u})\rangle_{\textbf{u}}=a^{-(n-1)}\int_{\mathbb{S}^{n-1}} \mathcal{R_{\psi}}f(\mathbf{u},b,a) \varphi(\mathbf{u}) d\mathbf{u}.
\end{equation}
\section{Ridgelet characterization of bounded subsets of $\mathcal S'_{0}(\mathbb R^{n})$}\label{section bounded sets}
This section is dedicated to prove a characterization of bounded subsets of $\mathcal S'_{0}(\mathbb R^{n})$ via the ridgelet transform. We begin the ensuing useful proposition. Note that \cite{helgason} $R(\mathcal{S}_{0}(\mathbb{R}^{n}))$ is a closed subspace of $\mathcal{S}_{0}(\mathbb{S}^{n-1}\times\mathbb{R})$. The open mapping theorem implies that $R:\mathcal{S}_{0}(\mathbb{R}^{n})\to R(\mathcal{S}_{0}(\mathbb{R}^{n}))$
is an isomorphism of topological vector spaces. We prove a similar result for the distributional Radon transform.
\begin{proposition} \label{radon iso}The Radon transform $R:\mathcal{S}_{0}'(\mathbb{R}^{n})\to R(\mathcal{S}_{0}'(\mathbb{R}^{n}))$ is an isomorphism of topological vector spaces.
\end{proposition}
\begin{proof} Since $R^{\ast}:\mathcal{S}_{0}(\mathbb{S}^{n-1}\times\mathbb{R})\to\mathcal{S}_{0}(\mathbb{R}^{n})$ is a continuous surjection between Fr\'{e}chet spaces, its transpose $R:\mathcal{S}_{0}'(\mathbb{R}^{n})\to \mathcal{S}'_{0}(\mathbb{S}^{n-1}\times\mathbb{R})$ must be continuous, injective, and must have weakly closed range \cite[Chap. 37]{treves}. The subspace $R(\mathcal{S}_{0}'(\mathbb{R}^{n}))$ is thus strongly closed because $\mathcal{S}'_{0}(\mathbb{S}^{n-1}\times\mathbb{R})$ is reflexive.  Pt\'{a}k's theory \cite{kotheII,rr} applies to show that $R:\mathcal{S}_{0}'(\mathbb{R}^{n})\to R(\mathcal{S}_{0}'(\mathbb{R}^{n}))$ is open if we verify that $\mathcal{S}_{0}'(\mathbb{R}^{n})$ is fully complete ($B$-complete in the sense of Pt\'{a}k) and that $R(\mathcal{S}_{0}'(\mathbb{R}^{n}))$ is barrelled. It is well known \cite[p. 123]{rr} that the strong dual of a reflexive Fr\'{e}chet space is fully complete, so $\mathcal{S}_{0}'(\mathbb{R}^{n})$, as a DFS space, is fully complete. Now, a closed subspace of a DFS space must itself be a DFS space. Since $\mathcal{S}'_{0}(\mathbb{S}^{n-1}\times\mathbb{R})$ is a DFS space, we obtain that $R(\mathcal{S}_{0}'(\mathbb{R}^{n}))$ is a DFS space and hence barrelled.
\end{proof}
We then have,
\begin{theorem}\label{boundedset}Let $\psi \in \mathcal S_{0}(\mathbb R)\setminus\{0\}$ and let $\mathfrak{B}\subset \mathcal
S'_{0}({{\mathbb R^n}})$. The following three statements are equivalent:
\begin{itemize}
\item [$(i)$] $\mathfrak{B}$ is bounded in  $\mathcal
S'_{0}({{\mathbb R^n}})$.
\item [$(ii)$] There are positive constants $l=l_{\mathfrak{B}}$ and $m=m_{\mathfrak{B}}$ such that for every $\varphi \in \mathcal {D}(\mathbb S^{n-1})$ one can find $C=C_{\varphi,\mathfrak{B}}>0$ with
\begin{equation}\label{bounded}|\langle \mathcal{R_{\psi}}f(\mathbf{u},b,a), \varphi(\mathbf{u}) \rangle_{\textbf{u}}|\leq C\left(a+\frac{1}{a}\right)^{l}(1+|b|)^{m}, \,\ \mbox{for all} \, \, \ (b,a)\in\mathbb{H} \mbox{ and } f \in \mathfrak{B}.\end{equation}
\item [$(iii)$] $\mathcal{R}_{\psi}(\mathfrak{B})$ is bounded in  $\mathcal
S'({{\mathbb Y^{n+1}}})$.

\end{itemize}
\end{theorem}

\begin{proof} By Proposition \ref{radon iso}, $\mathfrak{B}$ is bounded if and only if $\mathfrak{B}_{1}:=R(\mathfrak{B})$ is bounded in $\mathcal{S}'_{0}(\mathbb{S}^{n-1}\times \mathbb{R})=\mathcal{S}'_{0}(\mathbb{R},\mathcal{D}'(\mathbb{S}^{n-1}))$. On the other hand, in view of (\ref{rid-rad-wav}), the estimate (\ref{bounded}) is equivalent to one of the form
\begin{equation}\label{bounded2}|\langle \mathcal{W}_{\psi}h(\mathbf{u},b,a), \varphi(\mathbf{u}) \rangle_{\textbf{u}}|\leq C\left(a+\frac{1}{a}\right)^{s}(1+|b|)^{m}, \,\ \mbox{for all} \, \, \ h\in \mathfrak{B}_{1}.\end{equation}

$(i)\Rightarrow(ii)$. Assume that $\mathfrak{B}_{1}$ is bounded. As a DFS space, $\mathcal{D}'(\mathbb{S}^{n-1})$ is the regular inductive limit of an inductive sequence of Banach spaces, \cite[Prop. 3.2]{prof44} then implies the existence of $s=s_{\mathfrak{B}}$ and $m=m_{\mathfrak{B}}$ such that $(a+1/a)^{-s}(1+|b|)^{-m}\mathcal{W}_{\psi}(\mathfrak{B}_{1})$ is bounded in $\mathcal{D}'(\mathbb{S}^{n-1})$, which implies (\ref{bounded2}).

$(ii)\Rightarrow(iii)$. If the estimates (\ref{bounded}) hold, we clearly have that for fixed $\varphi\in\mathcal{D}(\mathbb{S})$ and $\Psi\in\mathcal{S}(\mathbb{H})$ the quantity $\langle \mathcal{R_{\psi}}f(\mathbf{u},b,a), \varphi(\mathbf{u}) \Psi (b,a)\rangle$ (see (\ref{eq extra})) remains uniformly bounded for $f\in\mathfrak{B}$. A double application of the Banach-Steinhaus theorem shows that $\mathcal{R}_{\psi}(\mathfrak{B})$ is a bounded subset of $L_{b}(\mathcal{S}(\mathbb{H}),\mathcal{D}'(\mathbb{S}^{n-1}))=:\mathcal{S}'(\mathbb{H},\mathcal{D}'(\mathbb{S}^{n-1}))$ ($=\mathcal
S'({{\mathbb Y^{n+1}}})$).

$(iii)\Rightarrow(i)$. Let $\eta\in\mathcal{S}_{0}(\mathbb{R})$. Since $\mathcal{R}^{t}_{\eta}$ is continuous, it maps $\mathcal{R}_{\psi}(\mathfrak{B})$ into a bounded subset of $\mathcal{S}'_{0}(\mathbb{R}^{n})$. That $\mathfrak{B}$ is bounded follows at once from the inversion formula (\ref{eqidentity}).

\end{proof}

\section{ Abelian and Tauberian theorems}\label{section abel-tauber}
In this last section we characterize the quasiasymptotic behavior of elements  of ${\mathcal S}'_{0}{\mathbf (}{\mathbb R}^{n}{\mathbf )}$ in terms of Abelian and Tauberian theorems for the ridgelet transform.
\subsection{Quasiasymptotics} We briefly explain in this subsection the notion of quasiasymptotics of distributions. For more detailed accounts, see the books \cite{estrada,PST,PSV,VDZ} (see also \cite{vindas1,vindasqb,Vindas4}). This notion measures the asymptotic behavior of a distribution by comparison with Karamata regularly varying functions \cite{seneta}. A measurable real-valued function, defined and positive on an interval of the form ${\rm (0,}A{\rm ]}$ (resp. ${\rm [}A{\rm ,}\infty {\rm ))}$, is called \textit{slowly varying } at the origin (resp. at infinity) if

\[\lim_{\lambda \rightarrow 0^{+}}\frac{L(a \lambda)}{L(\lambda)}=1 \ \ \  \left({\rm resp. } \lim_{\lambda \rightarrow \infty}\right) \ \ \ {\rm for \ \ each} \quad a>0.\]
Throughout the rest of the article $L$ stands for a slowly varying function at the origin (resp. at infinity). We say that the distribution $f\in {\mathcal S}'_{0}{\rm (}{{\mathbb R}}^n{\rm )}$ has \textit{ quasiasymptotic behavior} of degree $\alpha \in {\mathbb R}$ at the origin (resp. at infinity) with respect to $L$ if there exists $g\in {\mathcal S}_{0}'{\rm (}{{\mathbb R}}^n{\rm )}$ such that for each $\phi\in {\mathcal S}_{0}{\rm (}{{\mathbb R}}^n{\rm )}$
\begin{equation*}
\mathop{{\rm lim}}_{\lambda \to 0^{{\rm +}}}\left\langle
\frac{f\left(\lambda\mathbf{x}\right)}{{\lambda}^{\alpha }L\left(\lambda\right)}{\rm
,\ }\phi\left(\mathbf{x}\right)\right\rangle =\left\langle
g\left(\mathbf{x}\right),\phi\left(\mathbf{x}\right)\right\rangle \ \ \  \left({\rm resp. } \lim_{\lambda \rightarrow \infty}\right) .
\end{equation*}
We employ the following notation for the quasiasymptotic behavior:
\begin{equation} \label{quasieq1}f\left(\lambda\mathbf{x}\right){\rm \sim }{\lambda }^{\alpha }L{\rm (}\lambda {\rm )}g{\rm (}\mathbf{x}{\rm )} \ \ \ \mbox{as }\lambda\to0^{+} \ (\mbox{resp. }\lambda\to\infty ) \quad \mbox{in }\ \ {\mathcal S}'_{0}{\rm (}{{\mathbb R}}^n{\rm ),}
\end{equation}
which should always be interpreted in the weak topology of ${\mathcal S}'_{0}{\rm (}{{\mathbb R}}^n{\rm )}$. One can prove \cite{PSV} that $g$
must be homogeneous of degree $\alpha $, namely, $g\left(a{\mathbf x}\right){\rm =}a^{\alpha }g\left({\mathbf
x}\right)$, for each $a>0$.

Likewise, one can introduce quasiasymptotic boundedness \cite{vindasqb}. We write
\begin{equation} \label{quasieq2}f\left(\lambda\mathbf{x}\right)=O({\lambda }^{\alpha }L{\rm (}\lambda {\rm )})\ \ \ \mbox{as }\lambda\to0^{+} \ (\mbox{resp. }\lambda\to\infty )\quad \mbox{in } \ {\mathcal S}'_{0}{\rm (}{{\mathbb R}}^n{\rm ),}
\end{equation}
if the corresponding growth order bound holds after evaluation at each test function from $\mathcal{S}_{0}(\mathbb{R}^{n})$. All these notions admit obvious generalizations to vector-valued distributions (see e.g. \cite{DZ1,DZ2,prof44}). For example, we might consider quasiasymptotics of distributions from $\mathcal{S}'_{0}(\mathbb{R},\mathcal{D}'(\mathbb{S}^{n-1}))= \mathcal{S}'_{0}(\mathbb{S}^{n-1}\times \mathbb{R})$ with respect to the radial variable $p$.

\subsection{An Abelian result} We provide here an Abelian proposition for the ridgelet transform. The following simple but useful lemma connects the quasiasymptotic properties of a distribution with those of its Radon transform.
\begin{lemma}\label{radon quasiasymptotics lemma} $f\in{\mathcal S}'_0({\mathbb R})$.
\begin{itemize}
\item [$(i)$] $f$ has the quasiasymptotic behavior (\ref{quasieq1}) if and only if its Radon transform has the quasiasymptotic behavior
\[{R}f\left(\mathbf{u},\lambda p \right){\rm \sim }{\lambda }^{\alpha+n-1}L{\rm (}\lambda {\rm )\ }{R}g\left(\mathbf{u}, p \right)\ \ \ \mbox{as }\lambda\to0^{+} \ (\mbox{resp. }\lambda\to\infty )\ \ \ \mbox{in }\mathcal S_{0}'(\mathbb R, \mathcal D'(\mathbb S^{n-1})).\]

\item [$(ii)$] $f$ satisfies (\ref{quasieq2}) if and only if its Radon transform satisfies
\[{R}f\left(\mathbf{u},\lambda p \right)=O({\lambda }^{\alpha+n-1}L{\rm (}\lambda {\rm )})\ \ \ \mbox{as }\lambda\to0^{+} \ (\mbox{resp. }\lambda\to\infty )\ \ \ \mbox{in }\mathcal S_{0}'(\mathbb R, \mathcal D'(\mathbb S^{n-1})).\]

\end{itemize}

\end{lemma}
\begin{proof} Set $f_{\lambda}(\mathbf{x})=f(\lambda\mathbf{x})$. If $\varrho \in \mathcal S_{0}(\mathbb S^{n-1}\times \mathbb R)$, we have,

\begin{align*}\langle R{f_{\lambda}}(\mathbf{u}, p), \varrho(\textbf{u}, p)\rangle & = \frac{1}{\lambda^n}\langle f(\mathbf{x}), R^{\ast}\varrho(\textbf{x}/\lambda)\rangle
\\
&
= \frac{1}{\lambda^{n-1}}\left\langle  f(\mathbf{x}), \frac{1}{\lambda}\int_{\mathbb S^{n-1}}\varrho\left(\mathbf{u},\frac{\mathbf{x}\cdot \mathbf{u}}{\lambda}\right)d\mathbf{u}\right\rangle
\\
&=\frac{1}{\lambda^{n-1}}\langle R{f}(\mathbf{u}, \lambda p), \varrho(\textbf{u}, p)\rangle,
\end{align*}
namely, $Rf_{\lambda}(\mathbf{u},p)=\lambda^{-(n-1)}Rf(\mathbf{u},\lambda p)$.
The result is then a consequence of Proposition \ref{radon iso}.

\end{proof}



\begin{proposition}\label{abelRT} Suppose that $f\in \mathcal S'_{0}({\mathbb R})$ has the quasiasymptotic behavior (\ref{quasieq1}).
Then, given any $\varphi\in\mathcal{D}(\mathbb{S}^{n-1})$ and $(b,a)\in\mathbb{H}$, we have

\begin{equation}\label{abelian eq}\lim_{\lambda\to 0^{+}} \frac{\left\langle{{\mathcal R}}_{\psi}f\left(\mathbf{u}, \lambda b,\lambda a\right) ,\varphi(\mathbf{u}) \right\rangle_{\mathbf{u}}}{\lambda^{\alpha}L(\lambda)}= \left\langle {{\mathcal R}}_{\psi}g\left(\mathbf{u}, b,a \right),\varphi(\mathbf{u}) \right\rangle_{\mathbf{u}} \  \ \left({\rm resp. } \lim_{\lambda \rightarrow \infty}\right).
\end{equation}

\end{proposition}

\begin{proof} This proposition follows by combining Lemma \ref{radon quasiasymptotics lemma} and the relation (\ref{rid-rad-wav}) with the DFS-space-valued version of \cite[Prop. 3.1]{prof44} for the wavelet transform (see comments in \cite[Sect. 8]{prof44}).
\end{proof}
\begin{remark}
The limit (\ref{abelian eq}) holds uniformly for $(b,a)$ in compact subsets of $\mathbb{H}$.
\end{remark}
\begin{remark}
If $f\in\mathcal{D}'_{L^{1}}(\mathbb{R}^{n})$, then (\ref{abelian eq}) reads
$$
\int_{\mathbb{S}^{n-1}} \mathcal{R_{\psi}}f(\mathbf{u},\lambda b,\lambda a) \varphi(\mathbf{u}) d\mathbf{u}\sim \lambda^{\alpha+n-1}L(\lambda)\int_{\mathbb{S}^{n-1}} \mathcal{R_{\psi}}g(\mathbf{u},b,a) \varphi(\mathbf{u}) d\mathbf{u},
$$
as follows from (\ref{eq extra2})
\end{remark}

\subsection{Tauberian theorem} Our next goal is to provide a Tauberian converse for Proposition \ref{abelRT}. The next theorem  characterizes the quasiasymtotic behavior in terms of the ridgelet transform.

\begin{theorem}\label{tauberRT} Let  $\psi\in \mathcal S_0({{\mathbb R}})\setminus\{0\}$ and $f\in {\mathcal S}'_{0}({{\mathbb R}}^n)$. The following two conditions:

\begin{equation} \label{GrindEQ__5_4_}
\lim_{\lambda \to 0^+} \frac{1}{{\lambda}^{\alpha}L(\lambda)}\langle{{\mathcal R}}_{\psi}f\left(\mathbf{u}, \lambda b,\lambda a \right), \varphi(\mathbf{u})\rangle={{\rm M}}_{b,a }(\varphi) \ \ \ \left(\mbox{resp. }\lim_{\lambda\to\infty}\right)
\end{equation}

\noindent exists (and is finite) for every $\varphi \in \mathcal{D}(\mathbb S^{n-1})$ and $\left( b,a \right)\in {\mathbb H}\cap \mathbb{S}$ , and
there exist $m,l>0$ such that for every $\varphi \in \mathcal{D}(\mathbb S^{n-1})$
\begin{equation} \label{ogr}
\left|\langle{\mathcal R}_{\psi}f\left(\mathbf{u}, \lambda b,\lambda a \right), \varphi(\mathbf{u})\rangle_{\textbf{u}}\right|\le C_{\varphi} \lambda^{\alpha}L(\lambda ){\left(a+\frac{1}{a}\right)^{l}}{(1+|b|)^m}
\end{equation}
for all $\left(b,a \right)\in \mathbb{H}\cap\mathbb{S}$ and $0<\lambda < 1$ (resp. $\lambda> 1$) are necessary and sufficient for the existence of a distribution $g$ such that $f$ has the quasiasymptotic behavior (\ref{quasieq1}).

\end{theorem}

\begin{proof} Assume first that $f$ has the quasiasymptotic behavior (\ref{quasieq1}). Proposition \ref{abelRT} implies that (\ref{GrindEQ__5_4_}) holds with $M_{b,a}(\varphi)=\left\langle {{\mathcal R}}_{\psi}g\left(\mathbf{u}, b,a \right),\varphi(\mathbf{u}) \right\rangle_{\mathbf{u}}$. Set $f_{\lambda}(\mathbf{x})=f(\lambda\mathbf{x})$.  Using \eqref{eqdesingular}, one readily verifies the relation
\begin{equation}\label{rideps}\mathcal{R}_{\psi}f_{\lambda}(\mathbf{u}, b,  a)= \mathcal R_{\psi}f( \mathbf{u}, \lambda b, \lambda a).
\end{equation}
On the other hand, $f$ satisfies (\ref{quasieq2}). That (\ref{ogr}) must necessarily hold  follows from Theorem \ref{boundedset}.

Conversely, assume \eqref{GrindEQ__5_4_} and \eqref{ogr}. Applying the same argument as in the proof of \cite[Lem. 6.1]{prof44}, one may assume that they hold \emph{for all} $(b,a)\in\mathbb{H}$ (in the case of \eqref{ogr}, one may need to replace $l$ and $m$ by bigger exponents). We will show that there is $G\in \mathcal{S}'(\mathbb{Y}^{n+1})$ such that
\begin{equation}
\label{eq4.8}\lim_{\lambda \to 0^+} \left\langle\frac{{{\mathcal R}}_{\psi}f\left(\mathbf{u}, \lambda b,\lambda a \right)} {{\lambda}^{\alpha}L(\lambda)}, \Phi(\mathbf{u},b,a)\right\rangle=\langle G\left(\mathbf{u},b,a\right), \Phi(\mathbf{u},b,a)\rangle \ \ \left(\mbox{resp. }\lim_{\lambda\to\infty}\right)
\end{equation}
for each $\Phi\in\mathcal{S}(\mathbb{Y}^{n+1})$. Once (\ref{eq4.8}) had been established, the inversion formula (\ref{eqidentity}) would imply that (\ref{quasieq1}) holds with $g=(1/K_{\psi,\eta}) \mathcal{R}^{t}_{\eta} G$. Using Theorem \ref{boundedset} and (\ref{rideps}) again, the estimates (\ref{ogr}) are equivalent to the quasiasymptotic boundedness (\ref{quasieq2}), but also to the boundedness in $\mathcal{S}'(\mathbb{Y}^{n+1})$ of the set
\begin{equation}\label{set rid}
\left\{\frac{{{\mathcal R}}_{\psi}f\left(\mathbf{u}, \lambda b,\lambda a \right)} {{\lambda}^{\alpha}L(\lambda)}: \: 0<\lambda<1\right\} \ \ \ \left(\mbox{resp. } \lambda>1\right).
\end{equation}
By the Banach-Steinhaus theorem, the set (\ref{set rid}) is equicontinuous. It is thus enough to show that the limit in the left-hand side of (\ref{eq4.8}) exists for $\Phi$ in the dense subspace $\mathcal{D}(\mathbb{S}^{n-1})\otimes\mathcal{S}(\mathbb{H})$ of $\mathcal{S}(\mathbb{Y}^{n+1})$. So, we check this for $\Phi(\mathbf{u},b,a)=\varphi(\mathbf{u})\Psi(b,a)$ with $\varphi\in\mathcal{D}(\mathbb{S}^{n-1})$ and $\Psi\in\mathcal{S}(\mathbb{H})$.
The function $M_{b,a}(\varphi)$ occurring in (\ref{GrindEQ__5_4_}) is measurable in $(b,a)\in\mathbb{H}$ and, in view of (\ref{ogr}), is of slow growth, i.e., it satisfies
\[\left| M_{b,a}(\varphi)\right|\le C_{\varphi}{\left(a+\frac{1}{a}\right)^{l}}{(1+|b|)^{m}}, \quad {\rm for \ \ all} \left(b,a \right)\in {\mathbb H}.\]
So, employing (\ref{eq extra}) and the Lebesgue dominated convergence theorem, we obtain
\begin{align*}&\lim_{\lambda\to0^{+}}\left\langle\frac{{{\mathcal R}}_{\psi}f\left(\mathbf{u}, \lambda b,\lambda a \right)} {{\lambda}^{\alpha}L(\lambda)}, \varphi(\mathbf{u})\Psi(b,a)\right\rangle
\\
&
=
\lim_{\lambda\to0^{+}}\int^{\infty }_{0} \int^{\infty }_{-\infty } \left\langle\frac{{{\mathcal R}}_{\psi}f\left(\mathbf{u}, \lambda b,\lambda a \right)} {{\lambda}^{\alpha}L(\lambda)}, \varphi(\mathbf{u})\right\rangle \Psi(b,a)
\frac{dbda}{a}
\\
&
=\int^{\infty}_{0} \int^{\infty}_{-\infty} M_{b,a}(\varphi)\Psi(b,a)
\frac{dbda}{a}
\end{align*}
(resp. $\lim_{\lambda\to \infty}$). This completes the proof.
\end{proof}
The following fact was already shown within the proof of Theorem \ref{tauberRT}.
\begin{corollary} Let  $\psi\in \mathcal S_0({{\mathbb R}})\setminus\{0\}$ and $f\in {\mathcal S}'_{0}({{\mathbb R}}^n)$. Then, $f$ satisfies (\ref{quasieq2}) if and only if there are $m,l>0$ such that for every $\varphi \in \mathcal{D}(\mathbb S^{n-1})$ the estimate (\ref{ogr}) holds
for all $0<\lambda < 1$ (resp. $\lambda> 1$) and $\left(b,a \right)\in \mathbb{H}\cap\mathbb{S}$ (or, equivalently, $\left(b,a \right)\in \mathbb{H}$).
\end{corollary}

\end{document}